\documentclass[a4paper,11pt,english]{smfart}
\usepackage[cp1250]{inputenc}
\usepackage{amsmath,amssymb}
\usepackage{amsthm}
\usepackage{graphicx}
\usepackage{epstopdf}
\usepackage{xcolor}
\usepackage{multirow}
\usepackage{tabularx}
\usepackage{mathrsfs}
\usepackage{a4wide}

\def\di{\displaystyle}

\def\K{\mathbb{K}}
\def\R{\mathbb{R}}

\def\L{\mathcal{L}}

\newtheorem{definition}{Definition}
\newtheorem{lemma}{Lemma}
\newtheorem{theorem}{Theorem}
\newtheorem{remark}{\textbf{Remark}}

\begin{document}
\title[Construction of fractional derivatives]{Comments on various extensions of the Riemann-Liouville fractional derivatives : about the Leibniz and chain rule properties }

\author{Jacky Cresson and Anna Szafra\'{n}ska}
\date{}
\maketitle

\begin{abstract}
Starting from the Riemann-Liouville derivative, many authors have built their own notion of fractional derivative in order to avoid some classical difficulties like a non zero derivative for a constant function or a rather complicated analogue of the Leibniz relation. Discussing in full generality the existence of such operator over continuous functions, we derive some obstruction Lemma which can be used to prove the triviality of some operators as long as the linearity and the Leibniz property are preserved. As an application, we discuss some properties of the Jumarie's fractional derivative as well as the local fractional derivative. We also discuss the chain rule property in the same perspective.
\end{abstract}

\noindent



\tableofcontents

\section{Introduction}

Recently, many authors have tried to define new operators acting on continuous functions starting from the well known Riemann-Liouville derivatives. The reason for such generalizations is usually the fact that the Riemann-Liouville fractional derivatives do not respect some interesting properties. As an example, the Riemann-Liouville derivatives of a constant function is non zero or they do not satisfy the Leibniz's relation. The Jumarie's derivative introduced by G. Jumarie in 2006 (see \cite{j1}) and in 2009 (see \cite{j2}) and the local fractional derivatives introduced by Kolwankar and Gangal in 1996 (see \cite{kg}) are examples of such a strategy. The question is then : imposing some specific algebraic constraints on a given operator acting on continuous functions or a suitable subset of continuous functions, what type of operators can we construct ? \\

A partial answer was given by V.E. Tarasov in two articles dealing with operators satisfying the Leibniz or chain rule formula. Starting with a classical rigidity results for derivations acting on $C^2$ functions, he deduces that derivations over continuous functions can not satisfy the Leibniz or chain rule. The incompleteness of this argument was pointed out by X. Wang in \cite{wang} arguing that for some local version of the fractional derivative the Leibniz formula can hold and moreover that the aim of fractional derivative is precisely to deal with non-differentiable functions. However, some counter examples to the fact that the Jumarie fractional derivatives satisfies the Leibniz property were given in \cite{liu,tarasov1,tarasov2}. The same questioning concerning the chain rule property was also discussed.\\

In the following, we formulate precisely the extension problem and we prove in particular that any linear operator acting on the set of continuous functions and satisfying the Leibniz property is trivial (the obstruction Lemma). This result seems not to  be known by many authors working on fractional calculus so that we provide a proof in this article. This result means that in order to generalize the classical derivative of continuous functions, one has to cancel one of the two previous relations. If not, one obtains an operator which gives zero for all functions. This result gives a complete proof of Tarasov's idea that the Leibniz property is too strong to be preserved when extending the derivative to continuous functions. We give precise statements in the following.\\

We then discuss the construction of alternative fractional derivatives proposed by G. Jumarie and by Kolwankar and Gangal respectively. These two operators are supposed to satisfy the linearity and the Leibniz's property and to be defined on continuous functions. However, in view of the obstruction Lemma, this is not possible. We then look more precisely on the construction of these operators and their properties. We prove in particular that the Jumarie's fractional derivative can not satisfy the Leibniz's property. This result has strong consequences as it invalidates many uses of this operator as for example to extend the classical calculus of variations \cite{almeida,malino}.

We also discuss the Kolwankar-Gangal notion of local fractional derivatives. Here again, using the obstruction Lemma, one can prove that this operator must be trivial on a huge subset of continuous functions. This result was already proved using different arguments in \cite{bc2}.\\

We end our discussion with some way to bypass the obstruction Lemma. As we will see, even from the algebraic point of view, the Riemann-Liouville fractional derivatives can be seen as a universal solution to an extension problem over suitable analogue of continuous functions.

\section{Rigidity property and the Leibniz property}

In this Section, we consider the set $C^0 ([a,b])$ of real valued continuous functions defined on the interval $[a,b]$ and the set $C^1 ([a,b])$ of continuously differentiable real valued functions defined on $[a,b]$.\\

The classical derivative of a real valued function $f\in C^1 ([a,b])$ is usually denoted by $f'$ and defined for all $t\in ]a,b[$ as
\begin{equation}
\label{defderi}
f' (t)=\lim_{h\rightarrow 0} \di\frac{f(t+h)-f(t)}{h} .
\end{equation}
The classical derivative satisfies two fundamental properties:

\begin{itemize}
\item It is a {\it linear} operator: for all $(\lambda,\mu)\in \R^2$ and $(f,g)\in C^1 ([a,b])$, we have $(\lambda f +\mu g)' =\lambda f +\mu g $.

\item The {\it Leibniz} property : for all $(f,g)\in C^1 ([a,b])$, we have
\begin{equation}
(f\cdot g )' =f' \cdot g +f\cdot g' .
\end{equation}
\end{itemize}

These two properties have leaded to the algebraic notion of {\it derivations} which appears in many context.

\begin{definition}[Derivations]
Let $A$ be a ring over an algebraic closed field $\K$. A derivation on $A$ is a mapping $D : A\rightarrow A$ such that $D$ is $\K$ linear and satisfies the Leibniz property: for all $(x,y)\in A$, we have $D(x \cdot y ) = Dx \cdot y +x \cdot Dy$.
\end{definition}

We refer to (\cite{jacob},Chap.I,$\S$.2,p.5) for more details.\\

We denote by $D_{NL}$ the Newton-Leibniz derivation defined over $C^1 ([a,b])$ and given for all $x\in C^1 ([a,b])$ by $D_{NL} (x)=x'$. \\

A classical results is the following {\it rigidity} result:

\begin{lemma}[Rigidity]
\label{rigidity}
Any derivation $D$ over $C^2 ([a,b])$ is of the form
\begin{equation}
D=a.D_{NL} ,
\end{equation}
where $a$ is an arbitrary $C^1$ function.
\end{lemma}

This result is interesting when one is dealing with extending the classical derivative to more general functional space. Indeed, the linearity and the Leibniz property completely characterize this operator.\\

The usual proof of this result used the Hadamard Lemma (see \cite{gon}, Exercice 22, p.44-45). However, for the readers which are not familiar with such a type of result, a good idea is first to look for a derivation acting on polynomials in order to see in a constructive way the role of the Leibniz property.\\

Let $\R [x]$ be the set of polynomial with real coefficients. Let $P\in \R [x]$, $P(x)=\di\sum_{i=0}^n a_i x^i$, $a_i \in \R$, $i=0,\dots ,n$. We have
$D(P)=\di\sum_{i=0}^n a_i D(x^i)$. Moreover, the Leibniz property implies that $D(x^i)=ix^{i-1} D(x)$. As a consequence, we have
$D(P)=D(x) P'(x)$. The action of a derivation on polynomials is fixed by the choice of $D(x)$. The Leibniz property gives then huge constraints on the set of derivations.

\begin{remark}
\begin{itemize}
\item In \cite{tarasov1}, Tarasov deduces from this Lemma that any derivation must break the Leibniz formula in order to be define on more general functional sets. However, the conclusion of this Lemma is only that we have precisely rigidity of a derivation over $C^2 ([a,b])$ functions. As we will see, this results indicates that these two algebraic conditions are very strong and must be difficult to satisfy when extending the classical derivative. This is the content of our next Lemma.

\item The formulation of the "no violation of the Leibniz rule. No fractional derivative" by Tarasov is misleading. The statement is made for a one parameter $\alpha\in \R$ family of derivations $D^{\alpha}$  even if this parameter plays no role in the rigidity Theorem.

\item One can of course look for some rigidity property of some one parameter family of operators acting on continuous functions which are not differentiable. This can be done for example by dealing with tempered distributions and posing an algebraic problem which gives more or less a unique solution for the Riemann-Liouville fractional derivative (see \cite{ja2}).
\end{itemize}
\end{remark}

A more general result is proved by H. König and V. Milman in \cite{km}:

\begin{theorem}
\label{akmthm}
If $T:C^1 (\R )\rightarrow C(\R )$ is an operator satisfying the Leibniz property then there are continuous functions $c$, $d\in C(\R )$ such that
\begin{equation}
T(f)(x)=c(x)f'(x)+d(x)f(x)\ln \mid f(x)\mid .
\end{equation}
\end{theorem}

The main point is that the operator $T$ is not even assumed to be linear. The Leibniz property is by itself very strong. If one impose linearity then one obtains an extension of Lemma \ref{rigidity} as any derivations from $C^1 (\R )$ to $C(\R )$ is of the form $c(x)f' (x)$.

\section{What about derivations on continuous functions ?}

The previous result does not apply to derivation acting on continuous functions. This remark was in fact at the basic of many comments against Tarasov's reasoning concerning the impossibility to extend the classical properties of the derivatives to continuous functions. The authors argues that the main object of fractional calculus is precisely to deal with non differentiable functions and as a consequence the rigidity result is not sufficient to conclude to the non existence of such operators. These remarks are indeed correct and Tarasov's argument is indeed not complete. Unfortunately, one can prove that the Leibniz property together with linearity is a too strong condition if one wants the operator to be defined on continuous functions. Indeed, we have the following obstruction Lemma:

\begin{lemma}[Obstruction-Leibniz]
There exists no non trivial derivations over $C^0 ([a,b])$.
\end{lemma}

We give the proof for the convenience of the reader.

\begin{proof}
First, let $C_c :=c$, $c\in \R$ be a constant function. Then for all $c\in \R$ and any derivation $D$ over $C^0 ([a,b])$ one has $D(C_c )=0$. Indeed, the function $C_1$ is the neutral element of $(C^0 ([a,b],\cdot)$ so that for all
$f\in C^0 ([a,b])$, we have $f\cdot C_1 =f$. As a consequence, we have using the Leibniz property
\begin{equation}
D(C_1)=D(C_1 \cdot C_1)=D(C_1)\cdot C_1 +C_1 \cdot D(C_1) =2C_1D(C_1) ,
\end{equation}
from which we deduce that $D(C_1)=0$. Using the linearity, we have
\begin{equation}
D(C_c)=D(c\cdot C_1)=c\cdot D(C_1 )=0 .
\end{equation}
Let us consider a continuous function $f\in C^0 ([a,b])$. We can always find a continuous function $g\in C^0 ([a,b])$ such that $f=g^3$. As a consequence, we have
\begin{equation}
D(f)=D(g^3)=3g^2 \cdot D(g) .
\end{equation}
It follows that if there exists $t\in [a,b]$ such that $f(t)=0$ then $D(f)(t)=0$.

The end of the proof now goes as follows. Assume that for some $t\in [a,b]$, one has $f(t)\not= 0$. Let us define $g:=f-f(t)$. We have $g(t)=0$ and as a consequence $D(g)(t)=0$. As $D(g)=D(f)-D(C_{f(t)} )=D(f)$, we deduce that $D(f)(t)=0$.

Hence, for all $t\in [a,b]$, one has $D(f)(t)=0$ and the derivation is trivial.
\end{proof}

If one cancels the linearity condition, we have the following result proved in \cite{km}:

\begin{theorem}
If $T:C(\R ) \rightarrow C(\R )$ is an operator satisfying the Leibniz property then there exists a continuous function $d \in C(\R )$ such that $T$ has the form
\begin{equation}
T(f)(x)=d(x)f(x)\ln \mid f(x)\mid .
\end{equation}
\end{theorem}

This kind of functions is called an {\it entropy function} in \cite{km}.\\

At this point, we have a partial answer to the extension problem of the classical derivative to continuous functions. As we will see, it will be sufficient to discuss some proposed extensions in fractional calculus.\\

It must be pointed out that the obstruction Lemma is valid for all functional space $F\subset C^0 ([a,b])$ such that for all $f\in F$ one can find $g\in F$ such that $f=g^3$.

\section{About the Jumarie's fractional derivative}

In a series of papers, G. Jumarie introduces a "new" fractional derivative which is now called the {\it Jumarie's fractional derivative}. This derivative knows some success due to the fact that it satisfies unusual properties like the Leibniz's property or the Chain rule property and is moreover defined on the set of continuous functions. Precisely, following G. Jumarie (see \cite{j1}, Definition 2.2, p.1369):

\begin{definition}
Let $0<\alpha <1$, the Jumarie's fractional derivative denoted by $D_J$ is defined by
\begin{equation}
\label{jum1}
D_J [x] =\di\frac{d}{dt}
\left [
I^{1-\alpha}_{0+} [x-x(0)]
\right ] .
\end{equation}
\end{definition}

The previous definition is assumed to be define on continuous functions (see \cite{j1},Definition 2.1 p.1368 and top of p.1369). The author insists on the fact that continuous but non differentiable functions can be considered (see \cite{j1}, Introduction p.1367). We will return to these points in the following which will be fundamental.

\subsection{Triviality of the Jumarie's fractional derivative ?}

In (\cite{j1} p.1371, Equation (3.11) and \cite{j2} p.382, Corollary 4.1, Equation (4.3)) the author states without proof that his operator satisfies the Leibniz property
\begin{equation}
D_J^{\alpha} (x\cdot y) =D_J^{\alpha} (x) \cdot y +x\cdot D^{\alpha}_J (y),
\end{equation}
over the set of continuous functions (see \cite{j2}, discussion on the validity of each properties after Corollary 4.1, p.382).\\

Using the obstruction Lemma, we deduce :

\begin{lemma}
\label{jumfund}
Let us assume that the Jumarie fractional derivative $D_J^{\alpha}$, $0<\alpha <1$, satisfies the linearity and the Leibniz relation on the set of continuous functions then $D_J^{\alpha}$ is trivial, i.e. $D_J^{\alpha} :=0$.
\end{lemma}

A consequence of this result is that {\bf any work dealing with the Jumarie's fractional derivative and using the Leibniz property leads to an {\it empty} theory}.\\

However, the Jumarie's fractional derivative is known to be non zero on some special functions as for example monomial functions $t^{\gamma}$. What is the problem ?

\subsection{Jumarie's fractional derivative versus the Caputo derivative}

As the Jumarie's fractional derivative is obviously non trivial over some functions, it means that the assumptions of Lemma \ref{jumfund} are not satisfied. The linearity being evident, only the Leibniz property has to be questioned. In fact, one can easily proves the following Lemma :

\begin{lemma}
The Jumarie's fractional derivatives does not satisfy the Leibniz property.
\end{lemma}

This Lemma is a consequence of the fact that the Jumarie's fractional derivative can be rewritten as
\begin{equation}
D_J^{\alpha} [x] =D_{0+}^{\alpha} [x-x(0)]
\end{equation}
and the following well known equality (see \cite{kst}):

\begin{lemma}
For every $0<\alpha<1$ and $x\in AC ([a,b],\R )$, $D_{0,+}^{\alpha} [x]$ and ${}_c D_{0+}^{\alpha} [x]$ are defined almost everywhere on $[a,b]$ and the following equality holds:
\begin{equation}
{}_c D^{\alpha}_{0+} [x] =D_{0+}^{\alpha} [x -x(0)] .
\end{equation}
\end{lemma}

As a consequence, we have:

\begin{lemma}
For every $0<\alpha <1$ and every $x\in AC([a,b],\R )$, we have $D_J^{\alpha} [x]={}_c D^{\alpha}_{0+} [x]$.
\end{lemma}

This result implies the fact that the Jumarie's fractional derivative does not satisfy the Leibniz rule (as usual for the Caputo's derivative) and also that the Jumarie's derivative is not a new fractional derivative.

\section{The Kolwankar-Gangal local fractional derivative}

In \cite{kg1}, Kolwankar and Gangal introduce the idea of {\it local fractional derivative} giving rise to the local fractional calculus. The idea is to localize the Riemann-Liouville derivative with respect to its base point in order to recover a local operator. In \cite{bc}, Ben Adda and one of the authors have given an alternative representation of this theorem using what is now called $\alpha$ difference quotient local fractional derivative\footnote{The proof of the representation Theorem in \cite{bc} was incorrect as pointed out for example by \cite{cyz} and \cite{bd}. A correct proof was given in \cite{cyz} and later under more general conditions in \cite{bc2}.}. This result can be used to deduce that the Kolwankar-Gangal (KG) fractional derivative satisfies some special properties which are unusual on some subset of continuous functions. In particular, one can prove that the KG fractional derivative satisfies the Leibniz property. As for the Jumarie fractional derivative, this will have very strong consequences, in particular that this derivative is trivial over a very big set and almost everywhere trivial generically.

\subsection{Definitions and properties}

We denote by $\mbox{\rm D}^{\alpha}_{KG,\sigma} f$ the {\it Kolwankar-Gangal local fractional derivative} (KG-LFD) defined by
\begin{equation}
\mbox{\rm D}^{\alpha}_{KG,\sigma} [f] (y)=\lim_{x\rightarrow y^{\sigma}} D^{\alpha}_{y,\sigma} \left [ \sigma (f-f(y)) \right ](x) , \ \sigma =\pm .
\end{equation}
When $\alpha =1$, and $f\in AC ([a,b])$, we have $D^1_{y,\sigma} [f](x)=\sigma f'(x)$ and then $D^1_{y,\sigma} (\sigma (f-f(y))) (x) = \sigma^2 f' (x) =f' (x) \in L^1$. We deduce that for $f\in C^1$, $D^1_{KG,\sigma} [f] (y) =f'(y)$.\\

The domain of existence of the Kolwankar-Gangal fractional derivative is difficult to describe. First, the left Riemann-Liouville derivative of a function $f$ is defined as long as $f$ belongs to the functional space ${\rm E}^{\alpha}_{a,+} ([a,b])$ defined by
\begin{equation}
\mbox{\rm E}^{\alpha}_{a,+} ([a,b])=\left \{ f\in \mbox{\rm L}^1 ,\ I^{1-\alpha}_{a,+} [f] \in AC ([a,b]) \right \} .
\end{equation}
This space is defined in Samko and al. (see \cite{skm}, Definition 2.4,p.44).\\

Second, the Kolwankar-Gangal local fractional derivative makes sense as long as $\mbox{\rm D}_{a,+}^{\alpha} [f-f(a)]$ is well defined on a given interval $]a,a+\delta [$ where $\delta >0$. In the following, the {\it existence} of the KG-LFD will be always understood as $f$ being such that $f-f(a)\in E^{\alpha}_{a,+} ([a,a+\delta ])$. This condition implies that $\mbox{\rm D}_{a,+}^{\alpha} [f -f(a)]$ is well defined at least almost everywhere in $]a,a+\delta ]$. In particular, if $f-f(a)$ has a left RL fractional derivative in the usual sense, i.e. $I^{1-\alpha}_{a,+} [f-f(a)]$ is differentiable at every point, then $f-f(a)$ belongs to $E^{\alpha}_{a,+} ([a,a+\delta ])$ (see \cite{skm},Remark 2.2 p.44).\\

In \cite{bc} another quantity which generalizes directly the representation of the classical derivative as limit of a difference-quotient is defined by :
$$\mbox{\rm D}^{\alpha}_{BC, \sigma} [f] (y)=\Gamma (1+\alpha )\lim_{x\rightarrow y^{\sigma}} \di\frac{\sigma (f(x)-f(y))}{\mid x-y\mid^{\alpha}} ,\ \sigma =\pm .$$
This quantity was introduced by G. Cherbit in \cite{cher} and is denoted BC-LFD in the following. This quantity is called {\it difference-quotient} local fractional derivative\footnote{Up to the constant factor $\Gamma (1+\alpha )$.} in \cite{cyz} or local fractional derivative in the sense of Ben Adda-Cresson in \cite{dp}.\\

In \cite{bc}, we state the following Theorem :

\begin{theorem}
\label{main}
Let $0<\alpha <1$ and $f\in C^0 ([a,b])$ and $y\in ]a,b[$ be such that $\mbox{\rm D}^{\alpha}_{KG,\sigma} [f](y)$ exists, then
\begin{equation}
\mbox{\rm D}^{\alpha}_{KG,\sigma} [f] (y)= \mbox{\rm D}^{\alpha}_{BC,\sigma} [f](y) .
\end{equation}
\end{theorem}

For a complete and corrected proof of this theorem, we refer to \cite{bc2}. This equality leads easily to the following result :

\begin{lemma}
For all $(f,g)$ such that $\mbox{\rm D}^{\alpha}_{KG,\sigma} [f]$ and $\mbox{\rm D}^{\alpha}_{KG,\sigma} [g]$ exist, we have
\begin{equation}
D_{KG,\sigma}^{\alpha} [f\cdot g]=D_{KG,\sigma}^{\alpha} [f]\cdot g +f\cdot D_{KG,\sigma}^{\alpha} [g] .
\end{equation}
\end{lemma}

As for the Jumarie's fractional derivative, the fact that the Kolwankar-Gangal operator or BC- operator satisfy the Leibniz property seems to indicate that these operators must be trivial on some sets.

\subsection{Triviality of the Kolwankar-Gangal fractional derivative ?}

As already pointed out, the fact that the Kolwankar-Gangal fractional derivative satisfies the Leibniz property and is a linear operator gives some indications that this operator is trivial at least on some functional sets where it is defined. This is indeed the case, as proved in \cite{bc2}:

\begin{theorem}
\label{holder1}
Let $f\in H^{\lambda} ([a,b])$, then for all $0<\alpha <\lambda \leq 1$, we have $\mbox{\rm D}_{KG,+}^{\alpha} [f] (x)=0$ in $[a,b]$.
\end{theorem}

\begin{proof}
This follows from a direct computation. As $f\in H^{\lambda} ([a,b])$, $0<\lambda \leq 1$, we have for all $\alpha <\lambda$ and all $x\in [a,b[$ (see \cite{skm},p.239,Lemma 13.1 and page 242, Corollary of Lemma 13.2) that
\begin{equation}
D_{x,+}^{\alpha} [f] (y) =\di\frac{f(x)}{\Gamma (1-\alpha )\, (y-x)^{\alpha} } +\psi_f (y),
\end{equation}
where $\psi_f \in H^{\lambda -\alpha} ([x,x+\delta ])$ for a certain $\delta >0$ such that $\psi_f (x)=0$. As a consequence, $D_{x,+}^{\alpha} [f-f(x)] (x)=\psi_{f-f(x)} (x)=0$ which concludes the proof.
\end{proof}

\begin{remark}
The previous result is proved in (\cite{bd},Lemma 3.1 p.69) when $f$ is at least $C^1$.
\end{remark}

An alternative proof of Theorem \ref{holder1} is to see that the arguments of the obstruction Lemma holds when one is dealing with a functional set $F$ such that for all $f\in F$, we can find $g\in F$ such that $f=g^3$. This is indeed the case for Hölderian functions of order $1>\lambda>\alpha$.\\

The previous result does not implies that the Kolwankar-Gangal fractional derivatives is always trivial. One can indeed find some explicit class of functions for which this derivative can be effectively computed. See for example \cite{kl}.\\

Two results are important here and related to the validity of the Leibniz property for this operator and its domain of definition. \\

First, in order to apply the obstruction Lemma, the operator must satisfy the Leibniz property over a functional set where it is always defined. However, for Hölderian functions in the class $H^{\alpha} ([a,b])$ the RL fractional derivative of $f-f(x)$ is not always defined, even almost everywhere. Indeed, when $\lambda=\alpha$, we have $I^{1-\alpha}_{x,+} [f-f(x)] \in H^{1,1} ([x,x+\delta ])$ (see \cite{skm},Theorem 3.1 p.53-54) which corresponds to Log-Lipschitz continuous functions instead of $H^{1+(\lambda-\alpha)} ([a,b])$ when $\lambda >\alpha$.

\begin{remark}
An example of a function $f$ in $H^{\alpha} ([a,b])$ for which the RL fractional derivative of order $\alpha$ $D^{\alpha}_{x,+} [f-f(x)]$ is not defined is given by the Weierstrass function $W_{\alpha} (x)= \di\sum_{n=0}^{\infty} q^{-\alpha n} \cos ( q^n x)$ with $q>1$ (see \cite{ross} Theorem 2 p.150 and Remark 4 p.155).
\end{remark}

Second, assuming that the Kolwankar-Gangal fractional derivative is well defined on an interval and continuous, one can proved directly that it is trivial (see for example \cite{cr2},$\S$.3, Theorem 4.1 and Corollary 4.1 p.4923).

\begin{remark}
This negative result was in fact one of the reasons that has leaded one of the authors not to continue the exploration of the use of local fractional derivatives in Physics as first studied in \cite{bc3}.
\end{remark}

Even with a less stronger condition, existence almost everywhere on a given interval, we have the following result proved in \cite{cyz}:

\begin{theorem}
Let $f:[a,b]\rightarrow \R$ be a locally $\alpha$-Hölder continuous function in $[a,b]$ for $0<\alpha <1$. Assume that the Kolwankar-Gangal fractional derivative exist almost everywhere in $[a,b]$. Then, the Kolwankar-Gangal fractional derivative is zero almost everywhere in $[a,b]$.
\end{theorem}

As a consequence, the Kolwankar-Gangal fractional derivative is not trivial but almost everywhere trivial.

\section{Rigidity property and the chain rule property}

The previous analysis proves that the Leibniz property is indeed very strong. What about the chain rule property ? The Jumarie fractional derivative was assumed to satisfy some kind of chain rule formula (see \cite{j2},Corollary 4.1 formula (4.4) and (4.5)). However, as this fractional derivative corresponds exactly to the Caputo derivatives this can not be true. Tarasov has in particular discussed this problem restricting his attention to the action of an operator satisfying a specific chain rule formula on monomial functions (see \cite{tarasov2}). In fact, a complete answer is given in \cite{akm}. But first, we give the following result states in \cite{km}:

\begin{theorem}
Suppose $T:C^1 (\R )\rightarrow C(\R )$ satisfies the Leibniz rule and the chain rule functional equations
$$
\left .
\begin{array}{lll}
T(f\cdot g ) & = & Tf\cdot g +f\cdot Tg ,\\
T(f\circ g) & = & (Tf)\circ g \cdot Tg ,
\end{array}
\right .
$$
for all $f,g \in C^1 (\R )$. Then $T$ is either identically $0$ or the derivative, $Tf=f'$.
\end{theorem}

The chain rule property is even a stronger algebraic constraint than the Leibniz rule. Indeed, we have (see \cite{akm},Proposition 3):

\begin{theorem}
Assume that $T:C(\R )\rightarrow C(\R )$ satisfies the functional equation
$T(f\circ g)  =  (Tf)\circ g \cdot Tg$, for all $f$,$g \in C(\R )$ and that there exists $g_0 \in C(\R )$ and $x_0\in \R$ with $(Tg_0)(x_0 )=0$. Then $T$ is zero on the class of half-bounded continuous functions.
\end{theorem}

This Theorem implies the following result :

\begin{theorem}[Obstruction-chain rule]
There exists no non trivial operator $D : C(\R ) \rightarrow C(\R )$ which satisfies the chain rule property and such that $D$ is zero on constant functions.
\end{theorem}

The reason for formulating such a result is that most of the modifications of the fractional Riemann-Liouville derivatives are done in order to satisfy the fact that the derivative of a constant function is zero. We see that imposing this condition and the chain rule property lead to a trivial operator.\\

We return now to the "proof" proposed by G. Jumarie in \cite{j3} of the Leibniz rule for fractional derivatives as long as one considers non-differentiable functions. We have already seen that this is not possible but the previous Theorem invalidate also the approach given in \cite{j3}. Indeed, G. Jumarie assumes in \cite{j3} that his fractional derivative is zero on constant functions and satisfies the chain rule property (see \cite{j3},$\S$.1,p.50). As a consequence, even if one can make some computation assuming the existence of such an operator and to deduce some properties satisfied by this operator, all the theory is empty due to the fact that no such operators exist.

\section{Conclusion and perspectives}

Using some known results on operators satisfying the Leibniz or chain rule property, we have explained that a fractional derivative which as a basic constraint corresponds to the classical derivative over $C^1 (\R )$ function can not satisfy the chain rule property. The same is true for the Leibniz property as long as linearity is preserved.\\

The previous discussion can be extended in order to determine the set of algebraic conditions which are necessary in order to recover the classical Riemann-Liouville fractional derivatives. This will be discussed in an another work \cite{ja2}.

\end{document}